\documentclass[12pt,leqno]{amsart}
\makeatletter
\usepackage{amsfonts,delarray,amssymb,amsmath,amsthm,a4,a4wide}

\usepackage{color}

\title[Blow-up rate of the mean curvature and a gap theorem]{Blow-up rate of the mean curvature during the mean curvature flow and a gap theorem for self-shrinkers}

\author{Nam Q.  Le}

\address{Department of
Mathematics, Columbia University, New York,
 USA}
\email{namle@math.columbia.edu}
\author{Natasa Sesum$^{*}$}
\address{
Department of Mathematics,
Rutgers University,
New Jersey,
USA}
\email{natasas@math.rutgers.edu}

\thanks{$*:$ Partially supported
by NSF grant DMS-0905749}

\newcommand{\review}[2][\right]{\relax
\ifx#1\right\relax \left.\fi#2#1\rvert}
\let\abs=\envert
 \newtheorem{definition}{Definition}[section]
\newtheorem{theorem}{Theorem}[section]

\newtheorem{remark}{Remark}[section]
\newtheorem{corollary}{Corollary}[section]
\newtheorem{lemma}{Lemma}[section]
\newtheorem{cor}{Corollary}[section]
\newtheorem{claim}{Claim}[section]

\newcommand{\bef}{\begin{flushright}}
\newcommand{\eef}{\end{flushright}}

\newcommand{\eval}[2][\right]{\relax
\ifx#1\right\relax \left.\fi#2#1\rvert}

\let\abs=\envert

\numberwithin{equation}{section}
\let\norm=\enVert
\newcommand{\nn}{{\bf{n}}}
\newcommand\e{\varepsilon}

\newcommand{\h}{\hspace*{.24in}}
\newcommand{\dist}{\mathrm{dist}}

\newcommand{\vol}{\mathrm{vol}}

\def\h{\hspace*{.24in}}

\def\beq{\begin{eqnarray*}}
\def\eeq{\end{eqnarray*}}

\def\NN{\mbox{$I\hspace{-.06in}N$}}
\def\RR{\mbox{$I\hspace{-.06in}R$}}

\def\SS{{\bold S}}

\newenvironment{myindentpar}[1]%
{\begin{list}{}%
         {\setlength{\leftmargin}{#1}}%
         \item[]%
}
{\end{list}}

\begin{document}

\maketitle
\author
\pagenumbering{arabic}
\begin{abstract} 
In this paper, we prove that the mean curvature blows up at the same rate as
the second fundamental form at the first singular time $T$ of any compact, Type I mean curvature flow. 
For the
mean curvature flow of surfaces, we obtain similar result provided that the Gaussian density is less than two. Our proofs are based on continuous 
rescaling and the 
classification of self-shrinkers. We show that all notions of singular sets defined in \cite{St} coincide for any Type I mean curvature flow, thus 
generalizing the result of Stone who established that for any mean convex Type I Mean curvature flow.
We also establish a gap theorem 
for self-shrinkers. 
\end{abstract}
\noindent

\section{Introduction and main results}
Let $M^{n}$ be a compact $n$-dimensional hypersurface without boundary, and let $F_{0}: M^{n}\rightarrow \RR^{n+ 1}$ be a smooth immersion of $M^{n}$ into $\mathbb{R}^{n+1}$. Consider a smooth one-parameter family of embeddings
\begin{equation*}
 F(\cdot, t): M^{n}\rightarrow \RR^{n +1}
\end{equation*}
satisfying
$
 F(\cdot, 0) = F_{0}(\cdot)$ 
and 
\begin{equation}
 \frac{\partial F(p, t)}{\partial t} = -H(p, t)\nu(p,t), \,\,\,~\forall (p, t)\in M\times [0, T).
\label{MCF1}
\end{equation}
Here $H(p, t)$ and $\nu(p, t)$ denote the mean curvature and the outward unit normal for the hypersurface $M_{t} = F(M^{n},t)$ at $F(p, t)$, respectively.
We will sometimes also write $x(p, t) = F(p, t)$, $M_{0}= M$ and refer to (\ref{MCF1}) as to the mean curvature flow equation. 
The mean curvature vector is denoted by $ \overrightarrow{H} = -H\nu$. Furthermore, for any 
compact $n$-dimensional hypersurface $M^{n}$ which is smoothly embedded in $\RR^{n+1}$ by $F: M^{n}\rightarrow \RR^{n+1}$, 
let us denote by $g = (g_{ij})$ the induced metric where 
$g_{ij} = \langle \frac{\partial}{\partial x_{i}} F, \frac{\partial}{\partial x_{j}} F\rangle$, $A = (h_{ij})$ the second fundamental 
form where $h_{ij} = \langle \frac{\partial}{\partial x_{i}} \nu, \frac{\partial}{\partial x_{j}} F\rangle$,  $d\mu =\sqrt{\text{det}~(g_{ij})}~dx$ the volume form,
$\nabla$ the induced Levi-Civita connection.
Then the mean curvature of $M^{n}$ is given by $$H = g^{ij}h_{ij} = \text{div}~ \nu.$$
With our convention on the choice of the unit normal vector $\nu$, $H$ is $n/R$ on the $n$-sphere $\SS^{n}(R)$ of radius $R$ in $\RR^{n+1}$ and $H$ is $k/R$ on the cylinder
$\SS^{k}(R)\times \RR^{n-k}\subset \RR^{n+1}$ of radius $R$ for the spherical factor.\\
\h In \cite{LS2}, the authors 
established the blow up of the mean curvature $H$ at the first singular time of the mean curvature flow in the case of  type I singularities. This result
somewhat extends that of Huisken \cite{Huisken84} on the blow-up of the second fundamental form at the first singular time
of the mean curvature flow. Before stating this result, we first recall the following definition. 
\begin{definition}
We say that the mean curvature flow (\ref{MCF1}) is of type I at the first singular time $T < \infty$, if the blow-up rate of the curvature satisfies an upper bound of the form
\begin{equation}
\max_{M_{t}}\abs{A}^2(\cdot, t) \leq \frac{C_{0}}{T-t}, ~ 0\leq t<T, \,\,\, 
\label{typeI}
\end{equation}
for all $t\in [0,T)$.
\end{definition}
In \cite{LS2} we proved the following result:
\begin{theorem}\cite[Theorem 1.2]{LS2}
Assume (\ref{typeI}) for the mean curvature flow (\ref{MCF1}). Then
\begin{equation}
 \lim_{t\rightarrow T}\mathrm{max}_{M_{t}} \abs{H}^2(\cdot, t) =\infty.
\label{Hbound}
\end{equation}
\label{MCbound}
\end{theorem}
On the other hand, Huisken \cite{Huisken90} also gave the (sharp) lower bound on the blow-up rate of the second fundamental form at the first singular 
time. This lower bound was based on the
maximum principle and states that
\begin{equation}
 \mathrm{max}_{M_{t}} \abs{A}^2(\cdot, t) \geq \frac{1}{2(T-t)}.
\label{BLlower}
\end{equation}
Having had Theorem \ref{MCbound}, one can naturally ask if a similar statement like (\ref{BLlower}) also holds for the mean curvature. It turns out that the answer is
yes. In this paper, we 
prove that the mean curvature blows up at the same rate as
the second fundamental form at the first singular time $T$ of the mean curvature flow if all singularities are of type I.  
This is the content of the following result:
\begin{theorem}
Assume (\ref{typeI}) for the mean curvature flow (\ref{MCF1}). Then, at the first singular time $T$ of the mean curvature flow, there exists 
$C_{\ast}>0$ such that
\begin{equation}
 \limsup_{t\rightarrow T} \sqrt{T-t} ~\mathrm{max}_{M_{t}} H(\cdot, t)\geq C_{\ast}.
\label{Hrate}
\end{equation}
\label{mainthm}
\end{theorem}
Theorem \ref{mainthm} extends Theorem \ref{MCbound} in two directions: 
\begin{myindentpar}{1cm} 
$\bullet$ It gives a lower bound, optimal modulo constants, on the blow up rate for the mean curvature.\\
  $\bullet$ The bound here
has a sign, not just absolute value, meaning that $H^+\equiv\max\{H, 0\}$ blows up at the rate $(T-t)^{-1/2}$.
\end{myindentpar}
The result in Theorem \ref{mainthm} should be compared with its Ricci flow analogue. For type-I Ricci flow, Enders, M\"{u}ller and Topping \cite{EMT} obtained
a lower bound on the blow up rate for the scalar curvature at the first singular time of the Ricci flow, similar to (\ref{Hrate}). 
Their result and ours have been proved by blow-up arguments. Note that, in the Ricci flow,
we have a uniform lower bound for the scalar curvature and moreover, the scalar curvature of a complete gradient shrinking Ricci soliton (the limit of 
blow-ups of Ricci flow solution) is nonnegative. These statements have 
no analogues in the mean curvature flow. Therefore, the result obtained in (\ref{Hrate}) is interesting. However, it is not completely surprising if one 
observes the following somewhat analogous statements between the two flows:
\begin{myindentpar}{1cm}
 $\bullet$ There are no gradient shrinking Ricci solitons with scalar curvature negative somewhere.\\
$\bullet$ There are no self-shrinkers  with mean curvature negative everywhere.
\end{myindentpar}
The first statement follows by \cite{ChenRicci}. The latter statement follows from  \cite[Theorem 5.1]{Huisken93} for 
self-shrinkers with bounded second fundamental form, and 
Colding and Minicozzi \cite[Theorem 0.17]{CM} where no assumptions on the second fundamental form of the self-shrinkers were made. We will use 
the above observation as a replacement for the nonnegativity of the mean curvature of a self-shrinker (which is not always true) in our proof of 
Theorem \ref{mainthm}. \\
\h An easier version of Theorem \ref{mainthm}, for the purpose of illustration, is the following:
\begin{theorem}
Assume (\ref{typeI}) for the mean curvature flow (\ref{MCF1}). Then, at the first singular time $T$ of the mean curvature flow, there exists $C_{\infty}>0$ such that
\begin{equation}
 \limsup_{t\rightarrow T} (T-t)~ \mathrm{max}_{M_{t}} \abs{H}^2(\cdot, t)\geq C_{\infty}.
\label{MCrate}
\end{equation}
More generally, for any $\alpha\geq n$, there exists $C_{\alpha}>0$ such that
\begin{equation}
 \limsup_{t\rightarrow T} (T-t)^{\frac{\alpha-n}{2\alpha}} \norm{H}_{L^{\alpha}(M_{t})}\geq C_{\alpha}.
\label{MCratealpha}
\end{equation}
In the special case of $\alpha =n$, we obtain the following non-collapsing type result:
there exists $C>0$ such that
\begin{equation}
 \limsup_{t\rightarrow T} \norm{H}_{L^{n}(M_{t})}\geq C.
\end{equation}
 \label{BLR}
\end{theorem}
\h The proof of Theorem \ref{MCbound} was based on blow-up arguments using Huisken's monotonicity formula, the classification of self-shrinkers
and White's local regularity theorem for mean curvature flow. See also the recent paper \cite{Cooper} for a different approach, which does 
not give the blow-up rate
as in Theorems \ref{mainthm} and \ref{BLR}. The idea in the proof of Theorems \ref{mainthm} and \ref{BLR} and other results in the present paper is the
use of continuous rescaling. \\
\h In the case of the mean curvature flow of surfaces in $\RR^{3}$, without any assumptions on possible singularities, we proved in \cite{LS2} 
that if the Gaussian densities of the flow is  below two, then the mean
curvature must blow up at the first singular time.
In this paper, we 
sharpen the above result by establishing the blow-up rate, optimal modulo constants, of the mean curvature for mean curvature flow of surfaces with Gaussian densities below two. 
Equivalently, we will prove the following:
\begin{theorem}
 Let $M^{2}$ be a compact, smooth and embedded 2-dimensional manifold in $\RR^{3}$. \\
(a) Suppose that
\begin{equation}
  \lim_{t\rightarrow T}\mathrm{max}_{M_{t}} \abs{H}^2(\cdot, t) (T-t) =0.
\label{BL3d}
\end{equation}
Let $y_{0}\in \RR^{3}$ be a point 
reached by the mean curvature 
flow (\ref{MCF1}) at time $T$. If the Gaussian density at $(y_{0}, T)$ satisfies
\begin{equation}
 \lim_{t\nearrow T} \int \rho_{y_{0}, T} d\mu_{t}:= \lim_{t\nearrow T} \int \frac{1}{[4\pi(T-t)]^{n/2}} \mathrm{exp}(-\frac{\abs{y-y_{0}}^2}{4(T-t)})
d\mu_{t} < 2,
\label{below2}
\end{equation}
 then $(y_{0}, T)$ is a regular point of the mean curvature flow (\ref{MCF1}).\\
(b) The result in (a) is still valid if we replace (\ref{BL3d}) by the following weaker condition:
\begin{equation}
 \label{BL3d-sign}
\limsup_{t\rightarrow T} \sqrt{T-t}~ \mathrm{max}_{M_{t}} H(\cdot, t)\leq 0.
\end{equation}
\label{3dBLR}
\end{theorem}
In particular, our theorem says that for the mean curvature flows of surfaces with Gaussian densities below two, 
at the first singular time $T$, the mean curvature must blow up to infinity at the 
rate $(T-t)^{-\frac{1}{2}}$. \\
\h In \cite{EMT}, Ender, M\"{u}ller 
and Topping established the blow up rate of the scalar curvature at any singular point of type-I Ricci flow. 
In \cite{St}, Stone established the blow up rate of the second fundamental  form and the mean curvature at any singular point of the mean 
convex mean curvature flow having type-I singularities. 
In this paper, we remove the mean convexity condition in \cite{St} by establishing sharp blow-up rates of the mean curvature at any 
singular point of the Type I
mean curvature flow.
Before stating our result in that direction, we give the definitions of different types of singular points, as in \cite{St}.
\begin{definition}
\label{def-sing}
\begin{enumerate}
\item[(i)]
We say $p\in M_0$ is a special singular point of the flow (\ref{MCF1}), as $t\to T$, if there exists a fixed $\delta > 0$, such that, for some sequence of times $t_i\to T$,
$$|A|^2(F(p,t_i)) \ge \frac{\delta}{T - t_i}.$$
If, on the other hand, $|A|^2(p,t_i) \le \frac{C}{T-t_i}$, we say that $p$ is a type I special singular point, otherwise we say it is a type II special singular point.
\item[(ii)]
We say $p\in M_0$ is a general singular point of the flow (\ref{MCF1}), as $t\to T$, if there exists a fixed $\delta > 0$, such that, for some sequence of times $t_i\to T$, and some sequence of points $p_i\in M_0$, with $p_i\to p$,
$$|A|^2(F(p_i,t_i)) \ge \frac{\delta}{T-t_i}.$$
We distinguish between type I and type II general singular points as in the case of special singular points.
\end{enumerate}
\end{definition} 

Denote by $\Sigma_s$ the set of all special singular points of the flow and by $\Sigma_g$ the set of all general singular points of 
the flow.  Moreover, we denote by $\Sigma_A\subset \Sigma_s$ the set of all points $p\in M_0$ such that $|A|(F(p,t),t)$ blows up at the 
Type I rate as $t\to T$, that is $|A|^{2}(F(p,t),t) \ge \frac{\delta}{T-t}$ for all $t\to T$. Similarly, let $\Sigma_H$ be the set of all 
points $p\in M_0$ such that $|H|(F(p,t),t)$ blows up at the type I rate as $t\to T$. Let $\Sigma$ be the set of all points $p\in M_0$ that do not have a neighborhood $p\in U_p$ in which $|A(\cdot,t)|$ stays uniformly bounded as $t\to T$.

It is obvious that $\Sigma_H \subset \Sigma_A \subset \Sigma_s \subset \Sigma_g \subset \Sigma$. In \cite{St} it was proved 
that $\Sigma_s = \Sigma_g = \Sigma$, in the case of a mean convex flow ($H \ge 0$) and type I singular points. An analogous 
statement for the type I Ricci flow has been obtained in \cite{EMT}. Our goal in this paper is to show 
that $\Sigma_H = \Sigma$, that is all notions of singular sets coincide for any type I mean curvature flow, without requiring the mean convexity. 
Our result states as follows.

\begin{theorem}
\label{thm-all-sing-sets1}
Let $(M_t)$ be a closed, type I mean curvature flow in $\mathbb{R}^{n+1}$. Then $\Sigma_H = \Sigma$.
\end{theorem}
A consequence of Theorem \ref{thm-all-sing-sets1} is the following Corollary whose analogue has been proved for the type I Ricci flow in \cite{EMT}.
\begin{cor}
Consider the type I mean curvature flow (\ref{MCF1}). If $\mu_{0}(M_0) < \infty$, then $\lim_{t\to T}\mu_{t}(\Sigma) = 0$. Here $d\mu_{t}$ is the volume form
of $M_{t}$.
\label{zero-vol}
\end{cor}
For the case of the mean curvature flow with $H \ge -C$, having type-I singularities, we can 
prove a stronger statement. For this purpose, we define a special blow-up set $\Sigma_{H}^{\delta} \subset \Sigma_H$, as the set of all points $p\in M_{0}$
such that $H(F(p,t),t)\geq \frac{1}{\sqrt{(2+\delta)(T-t)}}$ for $t$ sufficiently close to $T$. Here $\delta$ is a given positive number.
We will prove the following result:
\begin{theorem}
\label{thm-all-sing-sets}
Let $(M_t)$ be a closed, type I mean curvature flow in $\mathbb{R}^{n+1}$, with $H \ge -C$ for all $t\in [0,T)$. Then $\Sigma^{\delta}_H = \Sigma$.
\end{theorem}
\begin{remark}
 For the shrinking spheres, we have $H(F(p,t),t)= \frac{1}{\sqrt{2(T-t)}}$. Thus the blow-up rate for the mean curvature in Theorem \ref{thm-all-sing-sets}
is sharp.
\end{remark}
\begin{remark}
 As can be seen from the proof, we can replace the lower bound $H\geq -C$ in Theorem \ref{thm-all-sing-sets} by
\begin{equation*}
 \min_{M_{t}} H(\cdot,t)\geq \frac{-c(t)}{\sqrt{T-t}}
\end{equation*}
where $c(t)\rightarrow 0$ as $t\rightarrow T$.
\end{remark}
 \h In this paper as well as in \cite{LS2}, the classification of self-shrinkers plays an important role. More relevant to our theorems is the question: under what conditions
can we conclude that a self-shrinker is a hyperplane? There are two commonly used conditions in the literature:
\begin{myindentpar}{1cm}
 $\bullet$ Any smooth self-shrinker with mean curvature zero must be a hyperplane \cite[Corollary 2.8]{CM}.\\
$\bullet$ Any self-shrinker with entropy sufficiently close to one (which is the entropy of the hyperplane) must be flat. This is Brakke's theorem \cite{Brakke}.
\end{myindentpar}
We offer another criterion in this paper. First, we recall the definition of a self-shrinker {\footnote{When a precise normalization in the definition
of a self-shrinker is not important, we can use different normalizations of a self-shrinker in this paper, especially 
in the proof of Theorem \ref{thm-all-sing-sets1}. This will make the notation 
less heavier.}} that we will use
in the statement of our gap theorem. A hypersurface $\Sigma$ is said to be 
a {\emph{self-shrinker}} if it 
satisfies the equation
\begin{equation*}
 H = \langle x,\nu\rangle.
\end{equation*}
Equivalently, a hypersurface is said to be a {\emph{self-shrinker}} if it 
is the time $t=-\frac{1}{2}$ slice{\footnote{In \cite{CM}, Colding and Minicozzi define self-shrinkers 
to be the time $t= - 1$ slice of a self-shrinking MCF; consequently, they get that $H = \frac{1}{2}\langle x , \nu \rangle$.}} 
of a self-shrinking mean curvature flow 
(MCF) that disappears at $(0,0)$, i.e., of a MCF satisfying  $M_t = \sqrt{-2t} \, M_{-\frac{1}{2}}$.  
Our gap result is concerned with self-shrinkers whose second fundamental forms have small norm:
\begin{theorem}
 If the hypersurface $\Sigma\subset \RR^{n+1}$ is a smooth complete embedded self-shrinker without boundary and with polynomial volume growth, and satisfies
$\abs{A}^{2}< 1$ then $\Sigma$ is a hyperplane.
\label{gap-boundA}
\end{theorem}
\begin{remark}
A similar gap result for Ricci soliton has been obtained by Munteanu and Wang \cite[Corollary 1]{MW}. Yokota \cite{Yo} obtained a gap theorem 
concerning the normalized f-volume
for gradient shrinking Ricci solitons.
\end{remark}
\begin{remark}
 The curvature bound $\abs{A}^{2}< 1$ in Theorem \ref{gap-boundA} is optimal. $\RR^{k}\times \SS^{n-k}(\sqrt{n-k})$, for $0<k<n$, are nonflat self-shrinkers with
$\abs{A}^{2} = 1.$
\end{remark}

\h The rest of the paper is organized as follows. In Section \ref{blowuprate}, we prove Theorems \ref{mainthm}, \ref{BLR}, and \ref{3dBLR}. In Section \ref{singsets} we prove 
Theorems \ref{thm-all-sing-sets1}, \ref{thm-all-sing-sets} and Corollary \ref{zero-vol}. The proof of theorem
\ref{gap-boundA} will be given in Section \ref{gapthm}.
 
\section{Blow-up rate of the mean curvature}
\label{blowuprate}
\h In this section, we will prove Theorems \ref{mainthm}, \ref{BLR} and \ref{3dBLR} on the blow-up rate of the mean curvature during the mean curvature flow 
having type-I singularities.
\begin{proof}[Proof of Theorem \ref{BLR}] We first prove (\ref{MCrate}).
We argue by contradiction. Suppose otherwise that
\begin{equation}
  \lim_{t\rightarrow T}\mathrm{max}_{M_{t}} \abs{H}^2(\cdot, t) (T-t) =0.
\label{BLratezero}
\end{equation}
Without loss of generality, assume that $M^{n}\subset B_{1}(0) \subset \RR^{n+1}$. Let $y_{0}\in \RR^{n+1}$ be a point reached by the mean curvature 
flow (\ref{MCF1}) at time $T$, that is, there exists a sequence $(y_{j}, t_{j})$ with $t_{j}\nearrow T$ so that $y_{j}\in M_{t_{j}}$ and 
$y_{j}\rightarrow y_{0}$. We show that $(y_{0}, T)$ must be a regular point of (\ref{MCF1}) and this will contradict the assumption that $T$ is the first
singular time. \\
\h Without loss of generality, assume that $y_{0}=0\in R^{n+1}$ is a singular point of the mean curvature flow. Then,
following Huisken \cite{Huisken90}, we define the rescaled immersions $\tilde{F}(p,s)$ by
\begin{equation}
 \tilde{F}(p,s) = (2(T-t))^{-1/2} F(p,t),~s(t) =-\frac{1}{2} \log (T-t).
\label{contscale}
\end{equation}
This is the {\it continuous rescaling} that is crucial in our proofs. The surfaces $\tilde{M}_{s} = \tilde{F}(\cdot,s) (M^{n})$ are therefore defined for $-\frac{1}{2}\log T \leq s<\infty$ and satisfy the equation
\begin{equation}
 \frac{d}{ds}\tilde{F}(\cdot, s)= -\tilde{H}(\cdot, s) \tilde{\nu}(\cdot, s) + \tilde{F}(\cdot, s).
\end{equation}
In view of  (\ref{typeI}), the rescaled surfaces $\tilde{M}_{s}$ have bounded curvature. By the smoothness estimate \cite{EH}, one can prove estimates for all
higher derivatives of the second fundamental form
\begin{equation}
 \abs{\tilde{\nabla} \tilde{A}}^2 \leq C_{m}(C_{0})~\forall m\geq 1.
\label{regest}
\end{equation}
Furthermore, because $F(0, t)\rightarrow 0$ as $t\rightarrow T$, using (\ref{typeI}) again, we find that the term $\tilde{F}(0, s)$ remains bounded. This follows
from the estimate
\begin{equation*}
 \abs{F(0, t)} \leq \int_{t}^{T} \abs{H(0,\tau)} d\tau \leq \int_{t}^{T} \frac{n^{1/2} C_{0}}{ (T-\tau)^{1/2}} d\tau \leq C (2(T-t))^{1/2}.
\end{equation*}
Hence we have the convergence $\tilde{M}_{s_{j}}\rightarrow \tilde{M}_{\infty}$ for a sequence
of times $s_{j}\rightarrow \infty.$  \\
\h Now, let $\tilde{\rho}(x) = \frac{e^{-\frac{1}{2}\abs{x}^2}}{(2\pi)^{n/2}}$. Then Huisken's normalized monotonicity formula \cite{Huisken90} reads
\begin{equation}
\label{eq-huis-mon}
 \frac{d}{ds} \int_{\tilde{M}_{s}} \tilde{\rho} d\tilde{\mu}_{s} = - \int_{\tilde{M}_{s}} \tilde{\rho} \abs{\tilde{H}\tilde{\nu} - \tilde{F}^{\perp}}^{2} d\tilde{\mu}_{s}.
\end{equation}
Here $\tilde{F}^{\perp}(\cdot,s)$ is the normal 
component of the position vector $\tilde{F}(\cdot,s)\in \RR^{n+1}$ in the normal space of $\tilde{M}_{s}$ in $\RR^{n+1}$.
From this we arrive at the following inequality
\begin{equation*}
\int_{s_{0}}^{\infty} \int_{\tilde{M}_{s}} \tilde{\rho} \abs{\tilde{H}\tilde{\nu} - \tilde{F}^{\perp}}^{2} d\tilde{\mu}_{s}\leq C.
\end{equation*}
In view of the regularity estimate (\ref{regest}) and Huisken's monotonicity formula (\ref{eq-huis-mon}), every limiting hypersurface $\tilde{M}_{\infty}$ satisfies the equation
\begin{equation}
 \tilde{H}_{\infty} = <\tilde{x}_{\infty} , \tilde{\nu}_{\infty}>. 
\label{selfsim}
\end{equation}
On the other hand, by (\ref{BLratezero}), we have $ \tilde{H}_{\infty}\equiv 0$. 
Thus $\tilde{M}_{\infty}$ is a minimal cone; see \cite[Corollary 2.8]{CM}.
 Because $\tilde{M}_{\infty}$ is smooth, it is a hyperplane.
In other words, the rescaled surfaces $\tilde{M}_{s}$ converge to a hyperplane. \\
\h Let $\rho_{y_{0, T}}: \RR^{n+1}\times (-\infty, T) \to \RR$ be the backward heat kernel at $(y_{0}, T)$, i.e, 
\begin{equation}
 \rho_{y_{0}, T} (y, t) = \frac{1}{[4\pi(T-t)]^{n/2}} \mathrm{exp}(-\frac{\abs{y-y_{0}}^2}{4(T-t)}).
\label{BWH}
\end{equation}
Then, the monotonicity formula of Huisken \cite{Huisken90} says that
\begin{equation}
\frac{d}{dt}\int_{M_{t}} \rho_{y_{0}, T} d\mu_{t} = -  \int_{M_{t}} \rho_{y_{0}, T} \abs{H\nu - \frac{F^{\perp}}{2(T-t)}}^2d\mu_{t},
\label{Hmono}
\end{equation}
from which it follows that the limit $\lim_{t\rightarrow T} \int_{M_{t}} \rho_{y_{0}, T} d\mu_{t}$ exists. Via the rescaling (\ref{contscale}), we have
\begin{equation}
 \int_{M_{t}} \rho_{y_0, T} d\mu_{t} = \int_{\tilde{M}_{s}} \frac{e^{-\frac{\abs{x}^2}{2}}}{(2\pi)^{n/2}} d\tilde{\mu}_{s}\equiv 
 \int_{\tilde{M}_{s}} \tilde{\rho} d\tilde{\mu}_{s}
\label{densitychange}
\end{equation}
where $d\tilde{\mu}_{s}$ is the induced volume form on $\tilde{M}_{s}$.\\
\h Because the rescaled surfaces $\tilde{M}_{s}$ converge to a hyperplane, we get that
\begin{equation}
 \lim_{s\rightarrow\infty} \int_{\tilde{M}_{s}} \frac{e^{-\frac{\abs{x}^2}{2}}}{(2\pi)^{n/2}} d\tilde{\mu}_{s} =
 \int_{\RR^{n}} \frac{e^{-\frac{\abs{x}^2}{2}}}{(2\pi)^{n/2}} dx =1.
\label{density1}
\end{equation}
Note that $s\rightarrow \infty$ as $t\rightarrow T$. Combining (\ref{densitychange}) and (\ref{density1}), we obtain
\begin{equation}
 \lim_{t\rightarrow T}  \int_{M_{t}} \rho_{y_0, T} d\mu_{t} =1.
\end{equation}
This means that the Gaussian density of $M^{n}$ at $(y_0, T)$ is 1.
By White's regularity theorem \cite{White}, the second fundamental form $\abs{A}(\cdot, t)$ of $M_{t}$ is bounded as $t\rightarrow T$
and $(y_{0}, T)$ is a regular point. \\
\h Finally, we prove (\ref{MCratealpha}). 
We use the same notion as above and argue by contradiction. Suppose otherwise that
\begin{equation}
\lim_{t\rightarrow T} (T-t)^{\frac{\alpha-n}{2\alpha}} \norm{H}_{L^{\alpha}(M_{t})} =0.
\label{BL0alpha}
\end{equation} 
Via the rescaling (\ref{contscale}), we have
\begin{equation}
 (T-t)^{\frac{\alpha-n}{2\alpha}} \norm{H}_{L^{\alpha}(M_{t})} = 2^{\frac{n-\alpha}{2\alpha}} \norm{\tilde{H}}_{L^{\alpha}(\tilde{M}_{s})}.
\label{st-alpha}
\end{equation}
Again, note that $s\rightarrow \infty$ as $t\rightarrow T$. Thus, letting $s\rightarrow\infty$ in (\ref{st-alpha}) and using (\ref{BL0alpha}), we obtain
$\norm{\tilde{H}}_{L^{\alpha}(\tilde{M}_{\infty})} =0$. Hence $\tilde{H}_{\infty} =0$ on $\tilde{M}_{\infty}$. Now, arguing
 as in the proof of (\ref{MCrate}), we
obtain a contradiction.
\end{proof}
Let us make a few observations. Without using Corollary 2.8 in \cite{CM}, one can also argue as follows.
If $\tilde{H}_{\infty}\geq 0$ in (\ref{selfsim}), then Huisken \cite{Huisken93} proved that 
$ \tilde{M}_{\infty}$ is one of the following:
\begin{myindentpar}{1cm}
 (i) $\SS^{n}$\\
(ii) $\SS^{n-m} (\sqrt{n-m})\times R^{m}$\\
(iii) $\Gamma \times R^{n-1}$
where $\Gamma$ is one of the homothetically (convex immersed) shrinking curves in $\RR^{2}$ found by Abresch and Langer \cite{AL}.
\end{myindentpar}
If we know $\tilde{H}_{\infty} =0$ at one point in $\tilde{M}_{\infty}$ then the only possibility is $ \tilde{M}_{\infty} = \RR^{n}$. This is a kind of a rigidity result
for self-shrinkers with nonnegative mean curvature. Note that, in \cite{EMT}, in order to establish the  blow-up rate of the scalar curvature at any 
singular point of type-I Ricci flow,   
the following rigidity result for gradient shrinking
solitons, due to Pigola-Rimoldi-Setti, played an important role:
\begin{lemma}\cite[Theorem 3]{PRS09}\label{rigidity}
Let $(M^{n},g,f)$ be a complete gradient shrinking soliton $R_{ij} + \nabla_{i}\nabla_{j} f = \frac{1}{2}g_{ij}$. Then the
scalar curvature $R_g$ is nonnegative, and if there exists a point
$p\in M$ where $R_g(p)=0$, then $(M,g,f)$ is the Gaussian soliton,
i.e. isometric to flat Euclidean space $(\RR^n,g_{\RR^n}, e^{\frac{\abs{x}^2}{4}})$.
\end{lemma}
Because the scalar curvature of a complete gradient shrinking soliton is nonnegative, the blow-up rate in \cite{EMT} was established at each singular point of the Ricci 
flow. For the mean curvature flow, the mean curvature of the self-shrinker $\tilde{M}_{\infty}$ satisfying (\ref{selfsim}) can possibly be negative, and therefore a pointwise statement 
for the blow-up rate of the mean curvature needs to be argued differently (see Theorem \ref{thm-all-sing-sets1}).
\begin{proof}[Proof of Theorem \ref{mainthm}]
 We argue by contradiction. Suppose otherwise that
\begin{equation}
 \limsup_{t\rightarrow T} \sqrt{T-t}~ \mathrm{max}_{M_{t}} ~H(\cdot, t)\leq 0.
\label{Hrate2}
\end{equation}
We will prove that any point $y_{0}$ reached by our mean curvature flow at time $T$ must be a 
regular point of (\ref{MCF1}) and this will contradict the assumption that $T$ is the first
singular time. We use the same rescaling as in the proof of Theorem \ref{BLR} and obtain in the limit a self-shrinker $\tilde{M}_{\infty}.$ Using 
(\ref{Hrate2}), we find 
that our self-shrinker $\tilde{M}_{\infty}$ satisfies $\tilde{H}_{\infty} \leq 0.$ 
Moreover, by (\ref{typeI}), we know that the second fundamental form of $\tilde{M}_{\infty}$ is bounded. Let $L$ be the differential operator $
 L = \Delta_{\tilde{M}_{\infty}} + \abs{\tilde{A}_{\infty}}^2 - <x, \nabla_{\tilde{M}_{\infty}}>.$
Then Huisken \cite[Theorem 5.1]{Huisken93} ( see also Colding and Minicozzi \cite[Lemma 5.5]{CM}) showed that $L\tilde{H}_{\infty} = \tilde{H}_{\infty}$. From the nonpositivity of $\tilde{H}_{\infty}$ and 
Harnack inequality, we can conclude that $\tilde{H}_{\infty}$ is either strictly negative or identically zero. The first case could not happen which 
follows from the proof of  Huisken's classification result \cite[Theorem 5.1]{Huisken93}.  For the convenience of a reader we will sketch it here.\\
\h Let $e_{1}, \cdots, e_{n},\nu_{\infty}$ be an adopted orthonormal frame. If $\tilde{H}_{\infty} < 0$ everywhere on $\tilde{M}_{\infty}$ we can 
consider the quantity $\frac{|\tilde{A}_{\infty}|^2}{\tilde{H}_{\infty}^2}$. Then, simple calculation shows that
\begin{eqnarray*}
\Delta\left(\frac{|\tilde{A}_{\infty}|^2}{\tilde{H}_{\infty}^2}\right) &=& \frac{2}{\tilde{H}_{\infty}^4}\abs{(\tilde{h}_{\infty})_{ij} \nabla_l \tilde{H}_{\infty} - \nabla_l (\tilde{h}_{\infty})_{ij} \tilde{H}_{\infty}}^2 \\
&-& \frac{2}{\tilde{H}_{\infty}}\nabla_i\tilde{H}_{\infty}\nabla_i\left(\frac{|\tilde{A}_{\infty}|^2}{\tilde{H}_{\infty}^2}\right) + \langle x, e_i\rangle \nabla_i\left(\frac{|\tilde{A}_{\infty}|^2}{\tilde{H}_{\infty}^2}\right).
\end{eqnarray*}
We multiply the equation by $|\tilde{A}_{\infty}|^2 \rho$, where $\rho$ is the rescaled heat 
kernel $e^{-\frac{|x|^2}{2}}$. Integrating by parts yields to
$$\int_{\tilde{M}_{\infty}} \abs{\nabla\left(\frac{|\tilde{A}_{\infty}|^2}{\tilde{H}_{\infty}^2}\right)}^2\rho\, d\mu + 2\int_{\tilde{M}_{\infty}}\frac{|\tilde{A}_{\infty}|^2}{\tilde{H}_{\infty}^4}|(\tilde{h}_{\infty})_{ij}\nabla_k \tilde{H}_{\infty} - \nabla_i(\tilde{h}_{\infty})_{jk}\tilde{H}_{\infty}|^2\rho\, d\mu = 0.$$
Huisken shows that in this case a complete and embedded 
self shrinker $\tilde{M}_{\infty}$ has to be of the form $\SS^{n-m}{\sqrt{n-m}}\times\mathbb{R}^m$, for $0 \le m \le n$, in which case $\tilde{H}_{\infty} \ge 0$, which 
contradicts our assumption that $\tilde{H}_{\infty} < 0$ everywhere on $\tilde{M}_{\infty}$
Here, we have adopted the convention of the outward unit normal vector when talking about the 
mean curvature and geometric quantities defined with respect to the normal vector such as is the mean curvature of the hypersurfaces under consideration. 

Thus we are left with the case $\tilde{H}_{\infty} =0$. Therefore, 
$\tilde{M}_{\infty}$ is a hyperplane. Now arguing as in the proof of Theorem \ref{BLR}, we can conclude that
$(y_{0}, T)$ must be a 
regular point of (\ref{MCF1}).
\end{proof}
We will now prove Theorem \ref{3dBLR}  in which we restrict ourselves to the case when $n = 2$, but we allow all possible types of singularities 
to happen at a finite singular time $T < \infty$. We adopt the proof from \cite{LS2} to show that the blow up rate of the mean curvature 
at the first singular time must be $(T-t)^{-\frac{1}{2}}$. The proof of (a) is very similar to the proof of Theorem 1.5 in \cite{LS2}, except that
we use here the continuous rescaling. The proof of (b) is a bit different. Though we also use the continuous rescaling, our limiting
self-shrinkers does not necessary have bounded second fundamental form. Thus we have to be more careful when dealing with the classification issues.
Huisken's classification result \cite[Theorem 5.1]{Huisken93} does not apply.
Thanks to Colding-Minicozzi \cite[Theorem 0.17]{CM}, this is not a problem.  
For the reader's convenience we will include the detailed proof below.

\begin{proof}[Proof of Theorem \ref{3dBLR}]
In this proof, $n=2$. Without loss of generality, assume that $M^{2}\subset B_{1}(0) \subset \RR^3$. Let $y_{0}\in \RR^3$ be a point 
reached by the mean curvature 
flow (\ref{MCF1}) at time $T$, that is, there exists a sequence $(y_{j}, t_{j})$ with $t_{j}\nearrow T$ so that $y_{j}\in M_{t_{j}}$ and 
$y_{j}\rightarrow y_{0}$. We show that $(y_{0}, T)$ is a regular point of (\ref{MCF1}) provided that (\ref{BL3d}) and (\ref{below2}) are satisfied. \\
\h We can assume that $y_{0} =0.$ Then,
following Huisken \cite{Huisken90}, we define the rescaled immersions $\tilde{F}(p,s)$ by
\begin{equation}
 \tilde{F}(p,s) = (2(T-t))^{-1/2} F(p,t)\equiv \lambda(s) F(p, t),~s(t) =-\frac{1}{2} \log (T-t).
\label{contscale3d}
\end{equation}
The surfaces $\tilde{M}_{s} = \tilde{F}(\cdot,s) (M^{n})$ are therefore defined for $-\frac{1}{2}\log T \leq s<\infty$ and satisfy the equation
\begin{equation}
 \frac{d}{ds}\tilde{F}(\cdot, s)= -\tilde{H}(\cdot, s) \tilde{\nu}(\cdot, s) + \tilde{F}(\cdot, s).
\end{equation}
The induced volume form of $\tilde{M}_{s}$ is denoted by $\tilde{\mu}_{s}$.\\
 For any set $A\subset \RR^{n +1}$, let 
us define the parabolically rescaled measures at $(y_{0}, T)$:
\begin{equation*}
 \mu^{\lambda(s)} (A) = [\lambda (s)]^{-n} \mathcal{H}^{n}\lfloor \tilde{M}_{s} (\lambda(s)\cdot A).
\end{equation*}
Here $\mathcal{H}^{n}$ is the n-dimensional Hausdorff measure.
Now, let $\tilde{\rho}(x) = \frac{1}{(2\pi)^{n/2}}\exp (-\frac{1}{2}\abs{x}^2)$. Then Huisken's normalized monotonicity formula \cite{Huisken90} reads
\begin{equation}
 \frac{d}{ds} \int_{\tilde{M}_{s}} \tilde{\rho} d\tilde{\mu}_{s} = - \int_{\tilde{M}_{s}} \tilde{\rho} \abs{\tilde{H}\tilde{\nu} - \tilde{F}^{\perp}}^{2} d\tilde{\mu}_{s}.
\label{normalmono}
\end{equation}
Because $M$ is a compact, smooth and embedded 2-dimensional manifold in $\RR^{3}$, the following local area bound holds
\begin{equation*}
 \mathcal{H}^{2}(M\cap B_{R}(x)) \leq CR^{2}, \forall R>0, x\in \RR^{3}.
\label{AreaboundM}
\end{equation*}
Using Huisken's monotonicity formula \cite{Huisken90}, we can prove that (see, for example \cite[Lemma 2.9]{CM})
\begin{equation}
  \mathcal{H}^{2}(M_{t}\cap B_{R}(x)) \leq CR^{2}, \forall R>0, x\in \RR^{3}, 0\leq t<T.
\label{AreaboundMt}
\end{equation}
It follows that
\begin{equation}
\mu^{\lambda(s)} (B_{R}(x)) \leq C R^{2},\forall x\in \RR^{3}, R>0, -\frac{1}{2}\log T \leq s<\infty.
 \label{areabound}
\end{equation}

{\it Using the area bound (\ref{areabound}), and the normalized monotonicity formula (\ref{normalmono}), we 
can follow the proof of the Theorem on weak existence of blowups in Ilmanen \cite[Lemma 8, p. 14]{I1} to show 
that} 
there exists a subsequence of $\lambda(s)$ as $s\rightarrow\infty$ such that
$\mu^{\lambda(s)}\rightharpoonup \mu^{\infty}$ in the sense of 
Radon measures and the following
statements hold:\\
(a) (self-similarity) $\mu^{\infty}(A) = \lambda^{-n}\mu^{\infty} (\lambda\cdot A)$, for all $\lambda>0$\\
(b) (limit measure is a self-shrinker) $\mu^{\infty}$ satisfies
\begin{equation}
 \overrightarrow{H}(x) + S(x)^{\perp}\cdot x =0, ~\mu_{\infty}~ \mathrm{a. e.}~ x
\end{equation}
(c) Furthermore, Huisken's normalized integral converges
\begin{equation}
 \int \tilde{\rho} d\mu^{\infty} = \lim_{s\nearrow \infty} \int \tilde{\rho}d\mu^{\lambda(s)} \,\,\, .
\label{Huiskenconv}
\end{equation}
Note that, by Allard's Compactness Theorem \cite{S} and the fact that $\int_{B_{R}(x)} \abs{\overrightarrow{H}^{\lambda(s)}}^2$ is bounded 
for each $R>0$,  the 
Radon measure $\mu^{\infty}$ is
integer $2$-rectifiable, that is
\begin{equation*}
 d\mu^{\infty} = \theta(x)d\mathcal{H}^{2}\lfloor X_{\infty}
\end{equation*}
where $X_{\infty}$ is an $\mathcal{H}^{2}$-measurable, $2$-rectifiable set and $\theta$ is an $\mathcal{H}^{2}\lfloor X_{\infty}$-integrable, 
integer valued "multiplicity function". \\
\h {\it Now, using the same argument as in the proof of the $\RR^{3}$ Blow-up Theorem of Ilmanen \cite[p.29]{I2}, we can show that }
$X_{\infty}$ has to be smooth. Let us briefly explain the notations used in $(b)$. We follow the presentation used in the proof of Theorem 1.4 in \cite{LS2}
and for the sake of completeness, we include it here. \\
For a locally $n$-rectifiable Radon measure $\mu$, we define its $n$-dimensional approximate tangent plane $T_{x}\mu$ (which exists $\mu$-a.e x) by
\begin{equation*}
 T_{x}\mu (A) = \lim_{\lambda \rightarrow 0} \lambda^{-n} \mu (x + \lambda\cdot A).
\end{equation*}
The tangent plane $T_{x}\mu$ is a positive multiple of $\mathcal{H}^{n}\lfloor P$ for some $n$-dimensional plane $P$. Let $S: \RR^{n +1} \longrightarrow G(n +1, n)$
denotes the $\mu-$ measurable function that maps $x$ to the geometric tangent plane, denoted by $P$ above. An important
quantity is the first variation of $\mu$, defined by $\delta V_{\mu}(X):=\int div_{S(x)}X(x) d\mu(x)$ for $X\in C^{\infty}_{c}(\RR^{n +1}, \RR^{n +1})$.  Here $div_{S}X
= \sum_{i=1}^{n} D_{e_{i}}X.e_{i}$ where $e_{1}, \cdots, e_{n}$ is any orthonormal basis of $S$. We also denote by $S$
the orthogonal projection onto $S$ and thus $div_{S}X$ can be written as $S: DX$. Now, if the total first variation $\norm{\delta V_{\mu}}$ is a Radon measure and is
absolutely continuous with respect to $\mu$, then we can 
define the generalized mean curvature vector $\overrightarrow{H} = \overrightarrow{H}_{\mu} \in L^{1}_{\mathrm{loc}}(\mu)$ of $\mu$ as follows
\begin{equation}
 \int div_{S}X d\mu =\int -\overrightarrow{H}\cdot X d\mu
\label{firstvar}
\end{equation}
for all $X\in C_{c}^{\infty}(\RR^{n +1}, \RR^{n +1})$. For further information on geometric measure theory, we refer the reader to Simon's lecture notes \cite{S}.
Note that when $\mu$ is the surface measure of a smooth $n$-dimensional manifold $M$, the generalized mean curvature vector $\overrightarrow{H}_{\mu}$ of 
$\mu$ exists and is also the classical mean curvature vector of $M$. Therefore, we can apply (\ref{firstvar})
to $\mu^{\lambda(s)}$, which is the rescaled surface measure of the smooth manifold $\tilde{M}_{s}$.
From (\ref{firstvar}) and the definition of $\mu^{\lambda(s)}$, one sees that the mean curvature vector 
$\overrightarrow{H}^{\lambda(s)}$ of $\mu^{\lambda(s)}$ 
is $\frac{\overrightarrow{H}_{t}}{\lambda(s)}$ where $\overrightarrow{H}_{t}$ is the mean curvature vector of $M_{t}$ where $t = T - e^{-2s}$.
Recall that $\lambda (s) = (2(T-t))^{-1/2}$ and $s =-\frac{1}{2} \text{log} (T-t)$. \\
(a) By (\ref{BL3d}), we have
\begin{equation*}
 \limsup_{s\rightarrow\infty} \abs{\overrightarrow{H}^{\lambda(s)}} \leq \limsup_{s\rightarrow\infty}\mathrm{max}_{M_{t}}
 \abs{\overrightarrow{H}_{t}}[(2(T-t))^{1/2}] =0.
\end{equation*}

The lower semicontinuity of $\int\abs{H} d\mu$ asserts that, for any $x\in \RR^{3}$ and $R>0$
\begin{equation*}
 \int_{B_{R}(x)}\abs{\overrightarrow{H}_{\infty}} d\mu^{\infty}\leq \liminf_{s\rightarrow \infty}\int_{B_{R}(x)} 
\abs{\overrightarrow{H}^{\lambda(s)}} d\mu^{\lambda(s)}\leq 
 \int_{B_{R}(x)} \limsup_{s\rightarrow \infty} \abs{\overrightarrow{H}^{\lambda(s)}} d\mu^{\lambda(s)} =0.
\end{equation*}
Thus $\overrightarrow{H}_{\infty} =0$. Now, because $X_{\infty}$ is smooth, the weak mean curvature vector $\overrightarrow{H}_{\infty}$ coincides with 
the mean curvature vector in classical sense.  Thus we have a smooth solution $X_{\infty}$ that is a self-shrinker with $H = 0$ and therefore by 
\cite[Corollary 2.8]{CM}, it has to be a hyperplane. Furthermore
$\mu^{\infty}$ represents the surface measure of the plane $X_{\infty}$.  \\
By the Constancy theorem \cite[Theorem 41.1]{S}, $\theta$ is a constant. Thus by the convergence of Huisken's normalized integral (\ref{Huiskenconv}), we see that
\begin{equation*}
 \lim_{t\nearrow T} \int \rho_{y_{0}, T} d\mu_{t} = 
\lim_{s\nearrow \infty} \int \tilde{\rho}d\mu^{\lambda(s)} 
=\int \tilde{\rho} d\mu^{\infty}
= \int \tilde{\rho}\theta d\mathcal{H}^{2}\lfloor X_{\infty}
=\theta.
\end{equation*}
In the last equation, we have used that the Huisken's normalized integral of a plane is one.
By (\ref{below2}) and Proposition $2.10$ in \cite{White}, $1 \le \theta<2$. It follows from the integrality of $\theta$ that $\theta\equiv 1$. 
By White's regularity theorem \cite{White}, the second fundamental form $\abs{A}(\cdot, t)$ of $M_{t}$ is bounded as $t\rightarrow T$
and $(y_{0}, T)$ is a regular point. Thus, the 
flow can be extended 
past time $T$.\\
(b) Assume that (\ref{below2}) and (\ref{BL3d-sign}) hold. We adopt the notation from (a). Then we want to show that $(y_0,T)$ is a regular point. Assume as above, without losing any generality, that $y_0 = 0$. Rescale similarly and argue as in part (a)  to conclude that $\mu^{\infty}$ is the limit of the sequence of measures $\mu^{\lambda(s)}$ with connected supports $\tilde{M}_{s}$. Thus the 
support $X_{\infty}$ of $\mu^{\infty}$
is also connected. Because the mean curvature of $X_{\infty}$ is locally bounded, by Schatzle's constancy theorem (see, e.g., \cite[Theorem 3.1]{LS2}), we can conclude that
$\theta$ is a constant on $X_{\infty}$. 
Thus by the convergence of Huisken's normalized integral (\ref{Huiskenconv}), we see that
\begin{equation*}
 \lim_{t\nearrow T} \int \rho_{y_{0}, T} d\mu_{t} = 
\lim_{s\nearrow \infty} \int \tilde{\rho}d\mu^{\lambda(s)} 
=\int \tilde{\rho} d\mu^{\infty}
= \int \tilde{\rho}\theta d\mathcal{H}^{2}\lfloor X_{\infty}
\geq\theta.
\end{equation*}
Here we used the fact that
\begin{equation*}
 \int \tilde{\rho} d\mathcal{H}^{2}\lfloor X_{\infty}\geq 1
\end{equation*}
for any self-shrinker $X_{\infty}$. 
By (\ref{below2}) and Proposition $2.10$ in \cite{White},
\begin{equation*}
 1\leq \lim_{t\nearrow T} \int \rho_{y_{0}, T} d\mu_{t}<2.
\end{equation*}
It follows from the integrality of $\theta$ that $\theta\equiv 1$. Because the self-shrinker has multiplicity one, we must have the smooth convergence of $
\tilde{M}_{s}$ to $X_{\infty}.$ 
Note that the mean curvature 
$\tilde{H}_{s}$ of $\tilde{M}_{s}$ 
is $H_{t} (2 (T-t))^{1/2}$ where $H_{t}$ is the mean curvature of $M_{t}$ where $t = T - e^{-2s}$.
Thus, by (\ref{BL3d-sign}), we have
\begin{equation*}
 \limsup_{s\rightarrow\infty} \max_{\tilde{M}_{s}}\tilde{H}_{s} \leq 0.
\end{equation*}
It follows that $H_{\infty}\leq 0$ on $X_{\infty}$. By the classification result of Colding-Minicozzi \cite[Theorem 0.17]{CM}, $X_{\infty}$ must be a hyperplane.
Note that in Theorem 0.17 in \cite{CM}, no boundedness on the second fundamental form of $X_{\infty}$ is assumed. Now, we can conclude the proof
as in (a). 
\end{proof}
The method of the proof of Theorem \ref{3dBLR} also proves the following result:
\begin{cor}
 Let $M^{2}$ be a compact, smooth and embedded 2-dimensional manifold in $\RR^{3}$. If the Multiplicity One Conjecture of Ilmanen \cite[p. 7]{I1} holds then
 at the first singular time $T$ of the mean curvature flow, there exists 
$C_{\ast}>0$ such that
\begin{equation*}
 \limsup_{t\rightarrow T} \sqrt{T-t}~ \mathrm{max}_{M_{t}} ~H(\cdot, t)\geq C_{\ast}.
\end{equation*}
\end{cor}

\section{Singular sets}
\label{singsets}
In this section, we will prove Theorems \ref{thm-all-sing-sets1}, \ref{thm-all-sing-sets} and Corollary \ref{zero-vol}. 
We will be still dealing with the type I mean curvature flow, defined by (\ref{typeI}), such that
$$\limsup_{t\to T}(\max_{M_t} |A|^2(p,t) ) = +\infty.$$

In this section our goal is to extend Stone's theorem in \cite{St} about the characterization of 
singular sets of the mean curvature flow to any type I mean curvature flow (without requiring $H \ge 0$ as in \cite{St}). This will 
tell us that at every singular point of the type I mean curvature flow the second fundamental form and the mean curvature have to blow 
up at the rate $(T-t)^{-\frac{1}{2}}$. Note that the analogous characterization of singular sets for the type I Ricci flow 
has been recently obtained in \cite{EMT}.  In \cite{EMT} one of the main tools in proving this characterization was 
Perelman's pseudolocality theorem \cite[Theorem 10.3]{Pe}. In \cite{BLC} the pseudolocality theorem for the mean curvature has been proved 
which motivated us to prove Theorem \ref{thm-all-sing-sets1}, that is, the following:
\begin{theorem}
\label{thm-char-set}
Assume (\ref{typeI}) for the mean curvature flow (\ref{MCF1}). Then $\Sigma_H = \Sigma$.
\end{theorem}
In the case of mean convex mean curvature flow, we have a stronger result, that is Theorem \ref{thm-all-sing-sets}. The proof of this theorem is simple so we
give it here first.
\begin{proof}[Proof of Theorem \ref{thm-all-sing-sets}]
Due to  the inclusions
$$\Sigma^{\delta}_{H}\subset\Sigma_H \subset \Sigma_A \subset \Sigma_s \subset \Sigma_g \subset \Sigma,$$ 
it is enough to show that $\Sigma \subset \Sigma^{\delta}_H$.  Let $p\in \Sigma \backslash \Sigma^{\delta}_H$, meaning that there exists a sequence $t_i\to T$ so that
\begin{equation}
\label{eq-eps-i}
H(F(p,t_i),t_i)\le \frac{1}{\sqrt{(2+\delta)(T-t_i)}}.
\end{equation}
Without loss of 
generality, assume that $F(p, t_{i})\rightarrow 0$.
Then, using the blow-up argument as in the proof of Theorem \ref{BLR}, we get in the limit a smooth 
self-shrinker with $\tilde{H}_{\infty} (0)\leq\sqrt {\frac{2}{2 +\delta}}.$
Under the mean convexity assumption and the smoothness of the limit blow-up hypersurface, we know from Huisken's classification \cite{Huisken93} that the self-shrinker
must be $\SS^{n-m} (\sqrt{n-m})\times R^{m}$ ($0\leq m\leq n$). The mean curvature of these surfaces is  $\sqrt{n-m}$.
Thus the inequality $\tilde{H}_{\infty} (0)\leq \sqrt{\frac{2}{2 +\delta}}$ forces $\tilde{M}_{\infty}$ to be $\RR^{n}$. This implies that any
limit blow-up hypersurface at $0$ must be a hyperplane. Its Gaussian density is one and by White's regularity theorem \cite{White} the norm of the second 
fundamental form $|A|(\cdot,t)$ has to be uniformly bounded in a neighborhood of $p$ as $t\to T$.  This means $p \notin \Sigma$ and we obtain a
contradiction. Therefore, $\Sigma^{\delta}_H = \Sigma$.
\end{proof}
\h Before we start proving Theorem \ref{thm-char-set}, we recall the definition of local $\delta$-Lipschitz graph of 
radius $r_0$ and state the pseudolocality theorem from \cite{BLC}.  
\begin{definition}
An $n$-dimensional submanifold $M \subset \tilde{M}$ is said to be a local $\delta$-Lipschitz graph of 
radius $r_0$ at $p\in M$, if there is a normal coordinate system $(y_1, \dots, y_m)$ of $\tilde{M}$ 
around $p$ with $T_pM = span\{(\frac{\partial}{\partial y_1}, \dots, \frac{\partial}{\partial y_m}\}$, a vector 
valued function $F: \{y' = (y_1, \dots, y_m)  |  (y_1^2 + \dots y_m^2 < r_0^2\} \to \mathbb{R}^{m-n}$, with $F(0) = 0$, $|DF|(0) = 0$ such 
that $M\cap\{|y'| < r_0\} = \{(y',F(y'))  |  |y'| < r_0\}$ and $|DF|^2(y') = \sum_{i,\beta} \left(\frac{\partial F^{\beta}}{\partial y_i}\right )^2 < \delta^2$.
\end{definition}

\begin{theorem}[{\bf Chen, Yin \cite[Theorem 1.4]{BLC}}]
For every $\alpha > 0$ there exist $\e > 0$ and $\delta > 0$ with the following property. Suppose we have a 
smooth solution to the mean curvature flow $M_t\subset \mathbb{R}^n$ properly embedded in $B(x_0,r_0)$ for $t\in [0,T]$ with $0 < T \le (\e r_0)^2$. Assume that at time zero, $M_0$ is a local $\delta$-Lipschitz graph of radius $r_0$ at $x_0\in M_0$. Then we have an estimate of the second fundamental form,
$$|A|^2(x,t) \le \frac{\alpha}{t} + \frac{1}{(\e r_0)^2},$$
on $B(x_0,\e r_0)\cap M_t$, for any $t\in (0,T)$.
\label{pseudoMCF}
\end{theorem}

\begin{proof}[Proof of Theorem \ref{thm-char-set}]
We have that $\Sigma_H \subset \Sigma$. Assume $p\in \Sigma\backslash \Sigma_H$. Let $t_i\in [T-c_i,T)$ be such that
\begin{equation}|H|(p,t_i) \le \frac{\epsilon_i}{\sqrt{T-t_i}},
 \label{Hsmall}
\end{equation}
with $\epsilon_i \to 0$ and $\lambda _i = (T - t_i)^{-1/2} \to 
\infty$ as $i\to\infty$. Consider the rescaled  sequence $F_i(\cdot,t) = \lambda_i (F(\cdot, T + \frac{t}{\lambda_i^2}) - p)$. It has the property 
that $|A|^2_i(\cdot,t) = \frac{|A|^2(T+\frac{t}{\lambda_i^2})}{\lambda^2_i} \le \frac{C}{(-t)}$, due to condition (\ref{typeI}) and also  $\lim_{i\to\infty} |H_i|(0,-1) = 0$.
Due to Huisken's monotonicity formula and the smoothness estimates \cite{EH}, we can let $i\rightarrow\infty$ and get 
that the limiting hypersurface $M_{\infty}^t$ is a self-shrinker,
i.e, $M^{s}_{\infty} = \sqrt{-s}M^{-1}_{\infty}$ for all $s<0$, with 
$$|H_{\infty}|(0,-1) = 0 \,\,\, \mbox{and} \,\,\,  H_{\infty}(\cdot,s) = \frac{\langle F_{\infty}, \mu_{\infty}\rangle}{(-2s)}.$$

\begin{lemma}
\label{lem-nabla-H}
For all $s<0$, we have 
\begin{equation*}
H_{\infty}(0, s) = 0,~\nabla H_{\infty}(0, s) = 0,
\end{equation*}
where $\nabla$ is the Euclidean derivative.
\end{lemma}

\begin{proof}
At every point on the surface $M^{s}_{\infty}$, there is an orthonormal 
frame consisting of the outward unit vector $\nu$ and vectors $\{e_i^{\infty}\}_{1\le i \le n}$, lying in a tangential plane to the 
hypersurface at the point.
Recall that $|H_{\infty}|(0,-1) = 0$. If we differentiate $H_{\infty}(\cdot, s) = \frac{\langle F_{\infty}, \mu_{\infty}\rangle}{(-2s)}$ at $0$, in the 
tangential directions, we obtain
\begin{eqnarray*}
(-2s)\nabla_i H_{\infty} &=& \langle \nabla_i F_{\infty}, \nu_{\infty}\rangle + \langle F_{\infty}, \nabla_i\nu_{\infty}\rangle \\
&=& \langle e^{\infty}_i, \nu_{\infty}\rangle + \langle F_{\infty},\nabla_i\nu_{\infty}\rangle = 0,
\end{eqnarray*}
since $\langle \nu_{\infty}, e_i^{\infty}\rangle = 0$ and $F_{\infty}(0,s) = \sqrt{-s}F(0,-1) = 0$  (recall that $F_{\infty}(0,-1) = 0$, because 
$F_{\infty}(0,-1)$ is the position vector of the origin at time $-1$). This implies 
\begin{equation}
\label{eq-tang-der}
\nabla_i H_{\infty}(0,s) = 0,
\end{equation}
where $\nabla_i$ are the tangential derivatives to the limiting hypersurface at the origin. We claim that
$\nabla_{\nu} H_{\infty}(0,s) = 0$, where $\nabla_{\nu}$ is the derivative in the normal direction to the hypersurface. 
At the origin, we have
$$2(-s) \nabla_{\nu} H_{\infty} = \langle \nabla_{\nu}F_{\infty}, \nu_{\infty}\rangle + \langle F_{\infty}, \nabla_{\nu_{\infty}}\nu_{\infty}\rangle = \langle \nabla_{\nu}F_{\infty}, \nu_{\infty}\rangle,$$
since $F_{\infty}(0,s) = 0$, for $s < 0$. The hypersurface $M^{s}_{\infty}$ at the origin can be locally written as a graph
$F_{\infty}(x_1, \dots, x_n) = (x_1, \dots, x_n, u(x_1, \dots, x_n))$, with
$$\nu_{\infty}(x_1, \dots, x_n) = \frac{1}{\sqrt{1+|\nabla u|^2}}\cdot (\frac{\partial u}{\partial x_1}, \dots, \frac{\partial u}{\partial x_n}, -1).$$
Then,
\begin{eqnarray*}
\langle \nabla_{\nu_{\infty}}F_{\infty}, \nu_{\infty}\rangle &=& \langle (\nabla_{\nu_{\infty}} x_1, \dots, \nabla_{\nu_{\infty}} x_n, \nabla_{\nu_{\infty}}u), \nu_{\infty}\rangle \\
&=& \frac{1}{\sqrt{|\nabla u|^2 + 1}}\cdot \langle \left(\frac{\partial u}{\partial x_1}, \dots, \frac{\partial u}{\partial x_n}, \left(\frac{\partial u}{\partial x_1}\right)^2 + \dots \left(\frac{\partial u}{\partial x_n}\right)^2\right), \nu_{\infty}\rangle \\
&=& 0.
\end{eqnarray*} 
This implies 
\begin{equation}
\label{eq-norm-der}
\nabla_{\nu_{\infty}} H_{\infty}(0,s) = 0.
\end{equation}
Relations (\ref{eq-tang-der}) and (\ref{eq-norm-der}) conclude the proof of the Lemma.
\end{proof}

\begin{claim}
For every $\tilde{\e} > 0$ there exists an $r_0$ so that
$$|\langle F_{\infty}, \nu_{\infty}|(x,s) \le \tilde{\e}\cdot |F_{\infty}(x,s)|,$$
for every $x\in B(0,r)\cap M_{\infty}^s$, every $r \le r_0$ and $-1\leq s < 0$.
\label{main-claim}
\end{claim}
\begin{proof}
Let $\tilde{\e} > 0$ be a small number and let $r_0 = r_0(\tilde{\e})$ so 
that $|\nabla H_{\infty}(\cdot,-1)| < \tilde{\e}/2$ in $B(0,2r_0)\cap M^{-1}_{\infty}$. Here we have used Lemma \ref{lem-nabla-H} and as in there,
$\nabla$ is the Euclidean derivative. We find
\begin{equation}
\label{eq-intermed}
|H_{\infty}|(x,-1) \le |H_{\infty}|(0,-1) + \frac{\tilde{\e}}{2}\cdot \dist(0,x) \le \frac{\tilde{\e}}{2} r,
\end{equation}
for every $x\in B(0,r)\cap M_{\infty}^{-1}$ and $r \le r_0$,  where $\dist$ is the Euclidean distance.
Since on a self shrinker $H_{\infty}(\cdot,s) = \frac{H_{\infty}(\cdot,-1)}{\sqrt{-s}}$,  (\ref{eq-intermed}) yields to
\begin{equation}
\label{eq-s}
|H_{\infty}|(x,s) \le \frac{\tilde{\e} r}{2\sqrt{-s}}, ~x\in B(0,r)\cap M_{\infty}^s.
\end{equation}
Combining $H_{\infty}(\cdot,s) = \frac{\langle F_{\infty}, \nu_{\infty}\rangle(\cdot,s)}{(-2s)}$ with (\ref{eq-s}), we find 
$$|\langle F_{\infty}, \nu_{\infty}\rangle(\cdot,s)| \leq \sqrt{-s}\tilde{\e} r\leq \tilde{\e} r,$$
in $B(0,r)\cap M_{\infty}^s$, for every $r \le r_0$ and every $-1 \le s < 0$, which implies the Claim.
\end{proof} 
\h Fix any time slice $s\in [-1,0)$. For ease of notation, we suppress the superscript $s$ in $M^{s}_{\infty}$ when no confusion arises. By 
Lemma 7.1 in \cite{BLC} we have that the connected component of $B(0, r_{0})\cap M_{\infty}$ containing the origin can be written as a graph $\{(x,h(x')  |   |x'| < \frac{r_0}{96}\}$ and that
$$|D h|(x') \le \frac{36}{r_0}|x'|, \,\,\, x'\in B(0,\frac{r_0}{96}).$$
By Claim \ref{main-claim}, we have that 
\begin{equation}
\label{eq-layer}
|\langle F_{\infty}, \nu_{\infty}\rangle| < \tilde{\e} r, \,\,\, \mbox{in} \,\,\, B(0,r)\cap M_{\infty},
\end{equation}
for every $r \le r_0$. Since $\tilde{\e} < 1$ is a very small constant, we can conclude that our self-shrinker is very close 
to being a hyperplane around the origin, in $B(0,r_0)\cap M_{\infty}$. Moreover estimate (\ref{eq-layer}) forces that there is exactly 
one component of $B(0,\frac{r_0}{96})\cap  M_{\infty}$ in $B(0,\frac{r_0}{96})$.  To see that we can argue as follows. If there 
existed another component, call it $\gamma$, it would have to intersect  $\partial B(0,\frac{r_0}{96})$ in two points. Our condition (\ref{eq-layer}) would 
imply that for every $x\in B(0,\frac{r_0}{96})\cap M_{\infty}$ we have
$$|\langle F_{\infty},\nu_{\infty}\rangle(x)| < \tilde{\e} |F_{\infty}|(x).$$
This means at both intersection points the position vector $F_{\infty}$ is almost tangential to the hypersurface. Furthermore, since $\nu$ always stays the outward unit normal vector,
this condition would also force that, at one intersecting point we have an angle between the position vector $F_{\infty}$ and the normal 
vector $\nu_{\infty}$ measured in the counterclockwise direction being $\frac{\pi}{2} \stackrel{+}{-} \alpha$; and at the other 
intersection point, that angle would have to be either $\frac{3\pi}{2}\stackrel{+}{-}\beta$  or $-\frac{\pi}{2} \stackrel{+}{-}\beta$, for some 
small positive numbers $\alpha, \beta$. Since the angle between $F_{\infty}$ and $\nu_{\infty}$ is changing continuously along $\gamma$, there would exist a point  $q\in \gamma$ at which the position vector $F_{\infty}$ and the normal vector $\nu_{\infty}$ are collinear. This would imply
$$|\langle F_{\infty}, \nu_{\infty}\rangle(q)| = |F_{\infty}|(q),$$
which contradicts (\ref{eq-layer}).

Therefore it follows that
$$M_{\infty}\cap B(p,\frac{r_0}{96}) = \{(x',h(x'))   |   |x'| < \frac{r_0}{96}\},$$
with
\begin{equation}
\label{eq-close}
|Dh(x')| \le \frac{36}{r_0}|x'|, \,\,\, x'\in B(0,\frac{r_0}{96}).
\end{equation}

Let $\delta > 0$ and $\e > 0$ be as in the pseudolocality theorem \ref{pseudoMCF} for the mean curvature 
flow (Theorem 1.4 in \cite{BLC}).  Let $\tilde{r}_0 = \min\{\frac{r_0}{96}, \frac{r_0 \cdot \delta}{144}\}$. We still have 
that $\langle F_{\infty}, \nu_{\infty}\rangle| < \tilde{\e} |F_{\infty}|$ for $x\in B(0,\tilde{r}_0)$ and the same arguments as in the previous paragraph yield to $B(0,\tilde{r}_0)\cap M_{\infty}$ having only one component in $B(0,\tilde{r}_0)$ that is graphical, that is,
$$B(0,\tilde{r}_0)\cap M_{\infty} = \{(x,h(x') | |x'| < \tilde{r}_0\}.$$
Using (\ref{eq-close}) we find
$$|Dh(x')| \le \frac{36}{r_0}\cdot |x'|, \,\,\, x'\in B(p,\tilde{r}_0),$$
that is
$$|Dh(x')| \le \frac{36}{r_0}\cdot\tilde{r}_0 < \frac{\delta}{2}.$$
This means $M_{\infty}$ is a local $\delta/2$-Lipschitz graph of radius $\tilde{r}_0$  at the origin. \\
\h Now, taking $s\in [-1,0)$ into consideration and inspecting the above argument, we see 
that $M^{s}_{\infty}$ is a local $\delta/2$-Lipschitz graph of radius $\tilde{r}_0$  at the origin for all $s\in [-1,0)$, which follows from 
Claim \ref{main-claim}. In our application, we can just take $s= -(\epsilon \tilde{r}_{0})^2.$
Because of the smooth convergence of $M_{j}(\cdot, s)$ to $M_{\infty}^{s}$, by taking  $j \ge j_0$ large enough, the 
rescaled hypersurface $M_j\cap B(p,\tilde{r}_0)$ is a local $\delta$- Lipschitz 
graph of radius $\tilde{r}_0$ at $p$. Let $Q := \lambda^2_{j_0}$. Then by the pseudolocality theorem \ref{pseudoMCF}
applied to the mean curvature flow with rescaled initial hypersurface $M_{j_{0}}$, we can conclude that
\begin{equation}
|A_{j_0}|^2(x,\tau) \le \frac{\alpha}{\tau +\e^2\tilde{r}_0^2} + 
\frac{1}{(\e \tilde{r}_0)^2}, \,\,\,  \tau\in (-(\e \tilde{r}_0)^2,0), \,\,\, x\in \left(M_{j_0}\right)_{\tau}\cap B(p,\tilde{r}_0\e).
\label{pseudo2}
\end{equation}
Here, with a little abuse of notation, we have denoted by $(M^{n})_{\tau}= F_{j_0}(\cdot,\tau)(M^n)$. From our rescaling, we see that (\ref{pseudo2}) is equivalent to
$$|A(x,t)|^2 \le Q(\frac{1}{(t-T)Q + (\e\tilde{r}_0)^2} + \frac{1}{(\e \tilde{r}_0)^2}), \,\,\, \mbox{for all} \,\,\, t\in (T - \frac{\e^2\tilde{r}_0^2}{Q},T),$$
on the neighborhood $B(p,\frac{\e\cdot\tilde{r}_0}{Q})\cap M_t$. The bound for times $t < T - \frac{(\e\tilde{r}_0)^2}{Q}$ follows 
immediately from the type I condition (\ref{typeI}). This implies $p \notin \Sigma$ and we obtain contradiction. This concludes 
that $\Sigma = \Sigma_H$ finishing the proof of the theorem.
\end{proof}

Having Theorem \ref{thm-all-sing-sets1}, we can follow the arguments in \cite{EMT} to show the analogous statement for the mean curvature flow about the size of singular sets, stated in Corollary \ref{zero-vol}.

\begin{proof}[Proof of Corollary \ref{zero-vol}]
The proof is the same as for the Ricci flow in \cite{EMT} and for the convenience of the reader we sketch it below.

Define for $k\in \NN$
$$\Sigma_{H,k} := \{p\in M_0 | \abs{H}^2(F(p,t),t) \ge \frac{1/k}{T-t}, \,\,\, \forall t\in [T-1/k,T)\},$$
and $\Sigma_{H,0} = \emptyset$. 
Then by Theorem \ref{thm-all-sing-sets1}, we have 
\begin{equation*}
 \Sigma_{H,k}\subset \Sigma_H = \Sigma.
\end{equation*}
By the definition of the sets $\Sigma_{H,k}$, we have for all $t\geq T-\frac{1}{k}$ on $\Sigma_{H, k}$ 
$$\int_{T-1/k}^t H^2\, d s \ge \log(\frac{1/k}{T-t})^{1/k}$$
Recall that
$$\frac{d}{dt}\mu_t = -H^2\mu_t.$$
This implies, using the obvious fact that $\int^{t}_{T-\frac{1}{k}} H^{2} ds \leq \int_{0}^{t} H^{2} ds$ for all $t\geq 0$, that
\begin{eqnarray*}
\mu_t(\Sigma_{H,k}\backslash \Sigma_{H,k-1}) &=& \int_{\Sigma_{H,k}\backslash \Sigma_{H,k-1}} e^{-\int_0^t H^2\, d s}\, 
d\mu_0  \le \int_{\Sigma_{H,k}\backslash \Sigma_{H,k-1}} e^{-\int_{T-1/k}^t H^2\, d s}\, d\mu_0\\
&\leq & k^{1/k} (T-t)^{1/k} \mu_0(\Sigma_{H,k}\backslash \Sigma_{H,k-1}) \\ &\leq& 2 (T-t)^{1/k} \mu_0(\Sigma_{H,k}\backslash \Sigma_{H,k-1}).
\end{eqnarray*}
Here, we have used the inequality $k^{\frac{1}{k}}\leq 2$ for all $k\in \NN$. Therefore,
\begin{eqnarray*}
\mu_t(\Sigma) &=& \sum_{k=1}^{\infty}\mu_t(\Sigma_{H,k}\backslash \Sigma_{H,k-1})  \\
&\le& 2\sum_{k=1}^{\infty}(T-t)^{1/k}\mu_0(\Sigma_{H,k}\backslash \Sigma_{H,k-1}).
\end{eqnarray*}
Since $\sum_{k=1}^{\infty}\mu_0(\Sigma_{H,k}\backslash \Sigma_{H,k-1}) = \mu_0(\Sigma_H)  
\le \mu_0(M_0) < \infty$, we have
$$\lim_{t\to T} \mu_t(\Sigma) \le 2\lim_{t\to T} \sum_{k=1}^{\infty} (T-t)^{1/k}\mu_0(\Sigma_{H,k}\backslash \Sigma_{H,k-1}) = 0.$$
\end{proof}

\section{A Gap theorem for self-shrinkers}
\label{gapthm}
In this section, we prove Theorem \ref{gap-boundA}.

\begin{proof}[Proof of Theorem \ref{gap-boundA}]
Our proof follows Colding-Minicozzi \cite{CM} who obtained the following identity (see (9.42) there) for any self-shrinker $\Sigma^{'}$ 
without boundary and with polynomial volume growth, satisfying $H= 
\frac{1}{2}<x,\nu>$ and certain integrability conditions:
\begin{equation*}
 \int_{\Sigma^{'}}  |\nabla H|^2  \, e^{- \frac{|x|^2}{4} } d\mu_{\Sigma^{'}}  
	=  -\int_{\Sigma^{'}} H^2 \left(  \frac{1}{2} - |A|^2  \right)   \, e^{- \frac{|x|^2}{4} } d\mu_{\Sigma^{'}}. 
\end{equation*}
Changing the normalization to our self-shrinker $\Sigma$, we obtain
\begin{equation}
 \int_{\Sigma}  |\nabla H|^2  \, e^{- \frac{|x|^2}{2} }d\mu_{\Sigma}     
	=  -\int_{\Sigma} H^2 \left(  1 - |A|^2  \right)   \, e^{- \frac{|x|^2}{2} }d\mu_{\Sigma}   
\label{mainiden}
\end{equation}
and thus
 \begin{equation*}
 \int_{\Sigma}  |\nabla H|^2  \, e^{- \frac{|x|^2}{2} }d\mu_{\Sigma}    + \int_{\Sigma} H^2 \left(  1 - |A|^2  \right)   \, e^{- \frac{|x|^2}{2} }
d\mu_{\Sigma}   =0.
\end{equation*}
Using $\abs{A}^{2}< 1$, we deduce that $H\equiv 0$ and thus $\Sigma$ must be a hyperplane.\\
\h For reader's convenience, we will briefly indicate how all integrability conditions are satisfied and how to obtain (\ref{mainiden}). 
We will omit $d\mu_{\Sigma}$ in integrals. Let us define
the linear operator 
\begin{equation*}
 \mathcal{L} v = \Delta_{\Sigma} v- \langle x, \nabla_{\Sigma} v\rangle 
\equiv e^{\frac{\abs{x}^2}{2}} \text{div}_{\Sigma} (e^{-\frac{\abs{x}^2}{2}}\nabla_{\Sigma} v).
\end{equation*}
Then, on $\Sigma$, we have
\begin{equation}
 \mathcal{L} H +\abs{A}^{2} H =H.
\label{iden2}
\end{equation}
(see \cite[Theorem 5.1]{Huisken93} and also \cite[Lemma 5.5]{CM}.)\\
Furthermore, the operator $\mathcal{L}$ is self-adjoint in a weighted $L^{2}$ space with weight $e^{-\frac{\abs{x}^2}{2}}$. This means that if $u, v$ are
$C^{2}$ functions with
\begin{equation}
\int_{\Sigma} \left ( \abs{u\nabla v} + \abs{\nabla u}\abs{\nabla v} + \abs{u \mathcal{L} v}\right) e^{-\frac{\abs{x}^2}{2}} <\infty 
\label{iden3}
\end{equation}
then we get (see \cite[Corollary 3.10]{CM})
\begin{equation}
\int_{\Sigma} u (\mathcal{L} v) e^{-\frac{\abs{x}^2}{2}} = -\int_{\Sigma} \langle \nabla u,\nabla v\rangle e^{-\frac{\abs{x}^2}{2}}. 
\label{iden4}
\end{equation}
Now, if we differentiate the equation $H =\langle x, \nu \rangle$ in an orthonormal frame $e_{1}, \cdots, e_{n}$ on $\Sigma$ as in \cite[Theorem 4.1]{Huisken90}
and obtain $\nabla_{i} H = \langle x, e_{l}\rangle h_{li}$. By our assumption $\abs{A}^2<1$, we obtain
\begin{equation*}
 \abs{H}^2<n~~\text{and}~~ \abs{\nabla H}^2\leq \abs{A}^2 \abs{x}^2 \leq \abs{x}^2.
\end{equation*}
Combining the above inequalities with (\ref{iden2}) and the fact that $\Sigma$ has polynomial volume growth, we find that $\abs{H}^2, \abs{\nabla H}^2$
and $H \mathcal{L} H$ are in the weighted $L^{1}$ space with weight $e^{-\frac{\abs{x}^2}{2}}$, i.e., 
\begin{equation*}
 \int_{\Sigma} \left ( \abs{H\nabla H} + \abs{\nabla H}^2 + \abs{H\mathcal{L} H}\right) e^{-\frac{\abs{x}^2}{2}}<\infty.
\end{equation*}
Therefore, we can apply (\ref{iden4}) to  $u = v = H$ to get
\begin{equation*}
 \int_{\Sigma} \abs{\nabla H}^2 e^{-\frac{\abs{x}^2}{2}}
= -\int_{\Sigma} H (\mathcal{L}H) e^{-\frac{\abs{x}^2}{2}} = -\int_{\Sigma} H^{2} (1-\abs{A}^2)e^{-\frac{\abs{x}^2}{2}}.
\end{equation*}
This gives the desired identity (\ref{mainiden}).
\end{proof}

By Lemma 2.9 in \cite{CM}, if $(M_t)$ is the closed mean curvature flow with the initial hypersurface $M_0$ and if $\tau > 0$ is given, there exists a constant $V = V(M_0,\tau)$ so that
$$\vol(B_r(x_0)\cap M_t) \le Vr^n,$$
for all $t \ge \tau$ and all $x_0\in \mathbb{R}^{n+1}$. As a consequence of this volume bound, any self-shrinker 
that arises as a blow up limit of a closed mean curvature flow has a polynomial volume growth. An immediate corollary of this consideration 
and Theorem \ref{gap-boundA} is the following result:
\begin{corollary}
If $(M_t)$ is a self-shrinker that is a blown up limit of the closed mean curvature flow, such that there is a $t_0$, with 
$\sup_{M_{t_0}}|A|(\cdot,t_0) < 1$, then $M_t$ must be a hyperplane.
\end{corollary}

{} 

\end{document}